\documentclass[a4paper,12pt]{article}
\usepackage[utf8]{inputenc}
\usepackage[T1]{fontenc}
\usepackage{pstricks}
\usepackage{pstricks-add}
\usepackage[english]{babel} 

\usepackage{amsmath, amsfonts, amssymb, amstext, amsthm} 
\usepackage{graphicx, subfigure} 
\usepackage{graphicx, subfigure} 
\usepackage[affil-it]{authblk} 
\usepackage{tikz}
\usepackage[titletoc,title]{appendix}

\usepackage[colorlinks=true,bookmarks=true,citecolor=blue]{hyperref}

\def\R{\mathbb{R}}

\def\e{\varepsilon}

\def\a{\alpha}

\def\la{\lambda}
\def\d{\delta}
\def\k{\kappa}
\def\o{\omega}

\def\fp{f'(0)}

\def\vp{\varphi}

\def\us{\overline{u}}

\def\ut{\tilde{u}}

\def\lp{\left(}
\def\rp{\right)}
\def\lc{\left[}
\def\rc{\right]}
\def\lb{\left|}
\def\rb{\right|}
\def\lV{\left\Vert}
\def\rV{\right\Vert}

\def\ML{\mathcal{L}}
\def\MD{\mathcal{D}}

\def\MR{\mathcal{R}}
\def\MN{\mathcal{N}}

\newtheorem{theorem}{Theorem}
\newtheorem{prop}[theorem]{Proposition}
\newtheorem{lem}[theorem]{Lemma}

\theoremstyle{definition}

\newcounter{hyp}

\theoremstyle{remark}
\newtheorem{remark}{Remark}[section]

\title{Entire solution in cylinder-like domains for a bistable reaction-diffusion equation}
  \author{Antoine Pauthier%
  \thanks{e-mail: \texttt{apauthie@umn.edu}}}
\affil{\footnotesize School of Mathematics, University of Minnesota,\\  206 Church Street S.E, Minneapolis, MN 55455}

\begin{document}
  \maketitle

\begin{abstract}
We construct nontrivial entire solutions for a bistable reaction-diffusion equation in a class of domains that are unbounded in 
one direction. The motivation comes from recent results of Berestycki, Bouhours, and Chapuisat \cite{BBC} concerning 
propagation and blocking phenomena in infinite domains. A key assumption in their study was the "cylinder-like"
 assumption: their domains are supposed to be straight cylinders in a half space. The purpose of this paper is to
 consider domains that tend to a straight cylinder in one direction. 
 We also prove the existence of an entire solution for a one-dimensional problem with a non-homogeneous linear term.
%
%
%
\end{abstract}

  \tableofcontents

\section{Introduction}
The purpose of this paper is to construct a non-trivial entire solution for a bistable equation in an unbounded domain in one direction. 
The equation under study is the following parabolic problem with Neumann boundary condition:
\begin{equation}\label{Neumannpb}
\begin{cases}
 \partial_t u(t,x,y)-\Delta u(t,x,y) =  f(u), & \ t\in\R,\ (x,y)\in\Omega, \\
 \partial_\nu u(t,x,y) = 0, & \ t\in\R,\ (x,y)\in\partial\Omega.
\end{cases}
\end{equation}
The domain $\Omega\subset\R^{N+1}$ is supposed to be smooth enough and infinite in one direction. 
The boundary condition is of Neumann type, $\partial_\nu u := \overrightarrow{\nu}.\nabla u=0$ where $\overrightarrow{\nu}$ 
is the outward normal derivative to $\Omega$ in $\partial \Omega.$
To emphasize the role of this 
main direction we write the spatial coordinates are $(x,y)$ with $x\in\R$ and $y\in\R^N.$ The domain is bounded in the transverse direction $y$ and tends to 
a straight cylinder as $x$ goes to $-\infty:$ we suppose that for all $x\in\R$ there exists a non-empty bounded section $\omega(x)\subset\R^N$   
such that
\begin{equation}\label{domaine1}
\begin{cases}
 \Omega=\left\{ (x,y),\ x\in\R,\ y\in \omega(x)\right\}, \\
\omega(x) \underset{x\to-\infty}{\longrightarrow}\omega^\infty\subset\R^N
\end{cases}
\end{equation}
where $\omega^\infty$ is the bounded non-empty limit section. Hence the domain $\Omega$ tends to the limit cylinder $\Omega^\infty=\R\times\o^\infty$ as $x$ goes to 
$-\infty.$ We will give precise assumptions on this convergence later.
Throughout the paper, the non-linear term $f$ is assumed to be of bistable kind, \textit{i.e.} there exists $\theta\in(0,1)$ such that
\begin{subequations} \label{hypbistable}
\begin{gather}
 f\in C^3([0,1]),\ f(0)=f(\theta)=f(1)=0,\ f'(0),f'(1)<0  \label{hypbistable1}\\
 f(s)<0 \textrm{ for all }s\in(0,\theta),\qquad f(s)>0 \textrm{ for all }s\in(\theta,1). \label{hypbistable2}
\end{gather}
\end{subequations} 
Moreover, we assume that the invader is the state $u=1,$ that is
\begin{equation}\label{1plusstable}
 \int_0^1 f(s)ds>0.
\end{equation}
Let us recall (see \cite{FMcL,AW1D}) that for such a kind of non-linearity, there exists a unique two-uple $(c,\vp)\in\R\times C^5(\R)$ satisfying
\begin{equation}\label{bistableTW}
 \begin{cases}
  \vp''+c\vp'+f(\vp)=0. \\
  \vp(-\infty)=1,\ \vp(+\infty)=0,\ \vp(0)=\theta.
 \end{cases}
\end{equation}
The function $(t,x)\mapsto\vp(x-ct)$ is, up to translation, the unique planar travelling wave for reaction-diffusion equation of the form
\begin{equation}\label{bistable}
 \partial_t u-\Delta u= f(u),\ t\in\R,(x,y)\in\R^{N+1}.
\end{equation}

\paragraph{Motivations}
When the underlying domain has a spatial dependence,
travelling waves of the form (\ref{bistableTW}) no longer exist. However the notion of transition waves has been generalised 
in \cite{BHgeneralisedTW} by Berestycki and Hamel. The specific problem of generalised 
transition waves for a bistable equation in cylinders with varying cross section
has been treated by Berestycki, Bouhours and Chapuisat in \cite{BBC}. 
They considered the parabolic problem (\ref{Neumannpb})
in a domain satisfying (\ref{domaine1}).
However, they also made the following assumption on $\Omega:$
\begin{equation}\label{cylinder}
 \Omega \cap \left\{ (x,y)\in\R^{N+1},x<0\right\}=\R^-\times \o,\ \o\in\R^{N}.
\end{equation}
This assumption asserts that the domain is equal to a cylinder in the left half space. See fig. \ref{domaines} for examples.
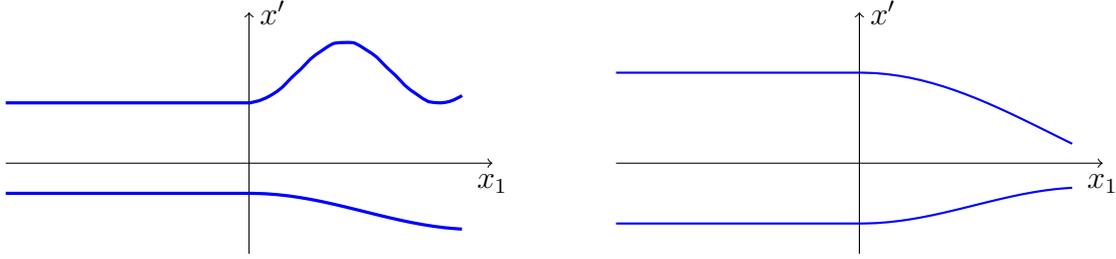
\begin{figure}[ht]
 \begin{tikzpicture}[scale=.8]
   \draw [->] (-4,0) -- (4,0);
   \draw [->] (0,-1.5) -- (0,2.5);
   \draw [blue,very thick] (-4,1) -- (0,1);
   \draw [blue,very thick] (-4,-0.5) -- (0,-0.5);
   \draw [blue,very thick,domain=0:3.5,samples=10,smooth,tension=1] plot (\x,{-0.5*cos(2*\x*180/pi)+1.5});
   \draw [blue,very thick,domain=0:3.5,samples=20,smooth,tension=1] plot (\x,{0.3*cos(\x*150/pi)-0.8});
   \draw (4,0) node[below] {$x_1$};
   \draw (0,2.5) node[right] {$x'$};
 \end{tikzpicture}
\hspace{1cm}
  \begin{tikzpicture}[scale=.8]
   \draw [->] (-4,0) -- (4,0);
   \draw [->] (0,-1.5) -- (0,2.5);
   \draw [blue, thick] (-4,1.5) -- (0,1.5);
   \draw [blue, thick] (-4,-1) -- (0,-1);
   \draw [blue, thick,domain=0:3.5,samples=20,smooth,tension=1] plot (\x,{cos(\x*90/pi)+0.5});
   \draw [blue, thick,domain=0:3.5,samples=20,smooth,tension=1] plot (\x,{-0.3*cos(\x*150/pi)-0.7});
   \draw (4,0) node[below] {$x_1$};
   \draw (0,2.5) node[right] {$x'$};
 \end{tikzpicture}
  \caption{\label{domaines}Two examples of domains $\Omega$ considered in \cite{BBC}}
\end{figure}
For such a kind of domains they obtained many properties concerning the propagation of a bistable wave, the first of them being
the existence and uniqueness of an entire solution: if $(c,\vp)$ is the unique solution of (\ref{bistableTW}), they showed that there exists a unique entire 
solution $u$ of (\ref{Neumannpb}) such that
 $$
 \lb u(t,x,y)-\vp(x-ct) \rb \underset{t\to-\infty}{\longrightarrow}0 \textrm{ uniformly in }\overline{\Omega}.
 $$
The assumption (\ref{cylinder}) was essential, but it is quite restrictive. It is therefore a relevant question to consider 
domains that \textbf{converge} to a cylinder as $x$ tends to $-\infty.$ This is the main topic of this paper.

\paragraph{The domain}
Throughout our study, we make the following assumptions on the domain.
We suppose that there exists a function 
\begin{equation}\label{hypdomain1}
 \Phi : \left\{ 
 \begin{array}{ccl}
  \Omega & \longrightarrow & \Omega^\infty \\
  (x,y) & \longmapsto & z=\Phi(x,y)
 \end{array}
\right.
\end{equation}
such that $\overrightarrow{\Phi} : (x,y) \mapsto \lp x,\Phi(x,y)\rp$ is a uniformly $C^3$ diffeomorphism between $\Omega$ 
and $\Omega^\infty$ with uniformly $C^3$ inverse. Moreover we suppose that $\Omega$ tends to $\Omega^\infty$ as $x\to-\infty$ in the sense 
that there exists $\k>0$ and some constant $C>0,$ for all $X\in\R,$
\begin{equation}\label{hypdomain2}
 \lV \overrightarrow{\Phi}-Id\rV_{C^2\lp \Omega\cap\{x<X\}\rp} \leq C\min\lp e^{\k X},1\rp.
\end{equation}
In particular, this assumption implies a uniform sliding sphere property for some radius $3r$ (see \cite{GT} for instance) on $\partial \Omega,$ which will be a key argument in our prove.
\begin{figure}[ht]
\begin{center}
      \begin{tikzpicture}[scale=0.7]
  \draw [->] (-4,0) -- (3,0);
  \draw [->] (0,-2.5) -- (0,2.5);
  \draw [red,dashed,thin] (-4,1.5) -- (0,1.5);
  \draw [red,dashed,thin] (-4,-1.5) -- (0,-1.5);
  \draw [blue,thick,samples=100,domain=-4:0] plot (\x,{1.5+1/(4+\x*\x)*cos(180*\x)});
  \draw [blue,thick,samples=100,domain=0:3] plot (\x,{2.75-cos(180*\x)+0.2*\x*sin(180*4*\x)});  
   \draw [blue,thick,samples=100,domain=-4:0] plot (\x,{-1.5+1/(2+\x*\x)});
  \draw [blue,thick,samples=100,domain=0:3] plot (\x,{-0.7-0.3*cos(100*\x)});  
     \draw (3,0) node[below] {$x$};
   \draw (0,2.5) node[left] {$y$};
  \end{tikzpicture}
  \end{center}
    \caption{\label{domainebis}An example of domain $\Omega$ considered in our paper}
\end{figure}
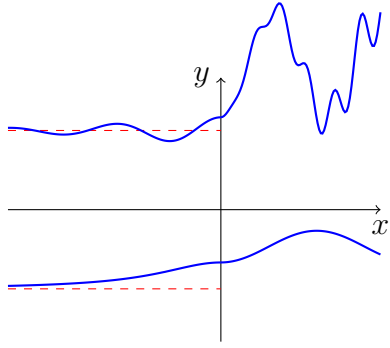
Contrary to what was done in \cite{BBC}, we consider a domain diffeomorphic to the straight cylinder. 
Hypothesis (\ref{hypdomain2}) asserts that the convergence to the straight cylinder is at some exponential rate and up to the second derivative.
However, it can be any exponential rate, hence it is a plausible guess that this assumption may be lightened.
An example of such a domain is given in figure \ref{domainebis}.

\paragraph{Results and organisation of the paper}
We prove the following existence theorem.

\begin{theorem}\label{thmgrosmodele}
 Under the assumptions (\ref{hypdomain1}) and (\ref{hypdomain2}) on the domain $\Omega$, there exists a function $u_\infty$ defined for $t\in\R,(x,y)\in\Omega$ solution 
 of (\ref{Neumannpb}) which satisfies
 \begin{equation}\label{convergencetowave}
 \sup \left\{ \lb u_\infty(t,x,y)-\vp(x-ct)\rb,\ (x,y)\in\Omega\right\}\underset{t\to-\infty}{\longrightarrow}0. 
 \end{equation}
 Moreover, the relation (\ref{convergencetowave}) uniquely determines $u_\infty.$
\end{theorem}

The proof of Theorem \ref{thmgrosmodele} amounts to proving stability of the travelling wave (\ref{bistableTW}) under the perturbation induced by the inhomogeneity of the domain.
Thus it is quite natural to first consider the following inhomogeneous problem in one dimension as a case study:
 \begin{equation}\label{toymodel}
  \partial_t u- \partial_{xx}u = f(u)\lp 1+g(x)\rp,\ t\in\R,x\in\R
 \end{equation}
 where $g$ is a bounded perturbation that satisfies the assumption
 \begin{equation}\label{hypothesesong}
 \textrm{there exists }\kappa>0,\  |g(x)|\leq e^{\kappa x} \textrm{ for all }x\in\R.
 \end{equation}
This model seems easier than the one considered previously. The non-homogeneous perturbation is only on the non-linearity and the problem is one dimensional.
Hypothesis (\ref{hypothesesong}) is closely related to our hypothesis (\ref{hypdomain2}); once again, we ask for an exponential rate, but it is arbitrary.
Existence of transition waves has already been studied for similar non-linearities (see \cite{zlatos15} for instance, and references therein).
However, as far as we know, the question of an entire solution that converges to the bistable wave as $t$ goes to $-\infty$ remained open.
It is the purpose of the next theorem.
 
 \begin{theorem}\label{solentieretoy}
 There exists a positive constant $\varpi$ which depends only on $f$ such that if $g$ satisfies (\ref{hypothesesong}) and 
 $g>-\varpi,$ then
 there exists a function $u_\infty=u_\infty(t,x)$ defined for $t\in\R,x\in\R$ solution of (\ref{toymodel}) which satisfies
 \begin{equation}
  \lV u_\infty(t,.)-\vp(.-ct)\rV_{L^\infty(\R)}\underset{t\to-\infty}{\longrightarrow}0.
 \end{equation}
 The constant $\varpi$ is given by $\varpi=\frac{\rho_1}{\lV f'\rV_\infty}$ where $\rho_1$ is the spectral gap 
 of the linearised operator associated with the travelling wave (\ref{bistableTW}).
\end{theorem}
 In the last paragraph of this introduction we recall some well-known facts about bistable non-linearities. Namely,
 the spectral decomposition associated to the linearised operator will be the most important tool of our study.
 The next section is concerned with the proof of Theorem \ref{solentieretoy}.
 It is done using a perturbative argument. We study the Cauchy problem associated to (\ref{toymodel}) starting from a translated bistable wave
 and prove that the solution can stay arbitrary closed to the wave up to a certain time, and then we use a compactness argument.
 The stability result is obtained thanks to a Lyapunov-Schmidt decomposition.
 We project the equation onto the kernel and the range of the operator.
 Equation on the kernel involves only quadratic terms, and
 the linear perturbation in the range is treated with an energy method. In order to get coercivity in the equation we need the hypothesis
 concerning the lower bound on the perturbation. Then we use a bootstrap argument between these two equations.
 
 The last section is concerned with the proof of Theorem \ref{thmgrosmodele}. The proof follows the same steps, and differs in one major
 point: perturbative terms coming from the right are controlled with travelling super-solution that we construct in the beginning of this last section.
 Then, the arguments are similar.

\paragraph{Bibliographical remarks}
Existence of travelling waves for reaction-diffusion equations have been studied since the well known works of Kolmogorov, Petrovskii and Piskunov \cite{KPP}.
It has been generalised in the late seventies by Aronson and Weinberger \cite{AW} and Fife and McLeod \cite{FMcL} for the specific problem of
a one dimensional bistable wave.
Since these pioneering works, an important effort has been done to study front propagation and transition waves in inhomogeneous
reaction-diffusion equations.
For our purpose let us mention
the paper of Berestycki and Nirenberg \cite{BNirenberg} where they showed existence of travelling waves in cylinders with an inhomogeneous 
advection field. The notion of transition fronts, or invasion fronts, has been introduced in \cite{Matano_talk} and precisely defined and studied in \cite{BHgeneralisedTW}.
The specific problem of transition waves when heterogeneities come from the boundary have received a much more recent attention.
In \cite{CG}, the authors showed that for a bistable nonlinearity, if the domain $\Omega$ is a succession of two half rectangular cylinders,
one can find conditions on their width for the propagation to be blocked. The problem of transition fronts for exterior domains has been
treated in \cite{BHM} where the authors devised geometrical conditions for the invasion to be complete. Finally, as already said, cylinder with varying 
cross sections has just been studied in \cite{BBC}.

Entire solutions for inhomogeneous one dimensional reaction-diffusion equations like (\ref{toymodel}) has been much studied in the few last years.
Existence for an ignition nonlinearity has been devised independently in \cite{MRS10} and \cite{NR09}. A generalisation for time and space inhomogeneous
reaction-diffusion equations of both ignition and bistable kind has recently been given in \cite{zlatos15}. However,
no asymptotic behaviour is given.

\paragraph{Acknowledgements} The research leading to these results has received 
funding from the European Research Council under the European Union’s 
Seventh Framework Programme (FP/2007-2013) / ERC Grant Agreement n.321186 - ReaDi -Reaction-Diffusion Equations, Propagation and Modelling.
This work was also partially supported by the French National
Research Agency (ANR), within the project NONLOCAL ANR-14-CE25-
0013.
I am grateful to Henri Berestycki and Jean-Michel Roquejoffre for suggesting me the
model and many fruitful conversations.


\paragraph{Some preliminary material}
\subparagraph{Behaviour at infinity}
As is shown in \cite{FMcL,GuoMorita}, it is easily seen that there exist two constants $C_1,C_2$ such that $\vp$ and $\vp'$ satisfy
\begin{equation}\label{estimeesbistable}
 \begin{cases}
C_1e^{\mu \xi}\leq 1-\vp(\xi) \leq C_2  e^{\mu \xi} & \xi\leq 0 \\
C_1e^{\la \xi}\leq \vp(\xi) \leq C_2  e^{\la \xi}  & \xi > 0  \\
C_1e^{\mu \xi}\leq -\vp'(\xi) \leq C_2  e^{\mu \xi} & \xi\leq 0 \\
C_1e^{\la \xi}\leq -\vp'(\xi) \leq C_2  e^{\la \xi} & \xi > 0  
 \end{cases}
\end{equation}
where
\begin{equation}\label{deflambdamu}
 \mu=\frac{-c+\sqrt{c^2-4f'(1)}}{2},\ \la=\frac{-c-\sqrt{c^2-4f'(0)}}{2}.
\end{equation}
Without lack of generality, we suppose that the exponential rate $\k$ given in (\ref{hypdomain2}) satisfies
\begin{equation}\label{hypdomain4}
 0<\kappa<-\la-\frac{c}{2}.
\end{equation}
This is just a technical assumption which is not restrictive.

\subparagraph{Spectral decomposition}
Throughout this paper, we will make a large use of a classical Lyapunov-Schmidt decomposition. We recall here some well-known facts of the involved spectral theory.
We consider the Banach space of bounded uniformly continuous functions that tend to 0 at infinity
$
\mathfrak{X}=BUC_0(\R).
$
 As we are looking for a stability result, 
 the linearised operator that will naturally appear in the moving framework $\xi=x-ct$ is given by
 \begin{equation*}
  \ML : \left\{ \begin{array}{rcl}
                 \MD(\ML)\subset \mathfrak{X} & \longrightarrow & \mathfrak{X} \\
                 v & \longmapsto & c\partial_\xi v+\partial_{\xi\xi}v+f'(\vp)v.
                \end{array}
                \right.
 \end{equation*}
It is a common result (see \cite{Sattinger} or \cite{Roquejoffre92} for instance) that there exists $\mathfrak{X}_1\simeq \MR(\ML)$
a closed subspace of $X$ such that
\begin{equation}\label{decompositionimagenoyau}
 \mathfrak{X}=\mathfrak{X}_1\oplus \MN(\ML).
\end{equation}
The null space of $\ML$ satisfies $\MN(\ML)=\MN(\ML^2)=\vp'\R.$ As we consider a bistable non-linearity, 0 is the first and an isolated eigenvalue
in the spectrum of $\ML.$ We denote $\rho_1$ the spectral gap between 0 and the second eigenvalue.
The projection on $\MN(\ML)$ is 
given by
\begin{equation}\label{definitioneetoile}
P\psi(\xi)=\langle e_*,\psi\rangle\vp'(\xi)=\frac{1}{\Lambda}\lp\int_\R e^{cz}\vp'(z)\psi(z)dz\rp\vp'(\xi)
\end{equation}
with the normalisation $\Lambda=\int_\R e^{cx}\vp'^2(x)dx.$ The projection on $\MR(\ML)$ is then given by
$$
Q\psi=\psi-P\psi.
$$
The operator $\ML_{|\mathfrak{X}_1}$ generates an analytic semi-group on $\mathfrak{X}_1$ endowed with the $L^\infty$ norm that
satisfies for all $t\geq0$
\begin{equation}\label{decML}
 \lV e^{t\ML}\rV  \leq C e^{-\rho t}
\end{equation}
where $\rho$ is any positive constant smaller than the spectral gap $\rho_1$ of $\ML$ and $C$ is a positive constant.

\section{The one-dimensional model: proof of Theorem \ref{solentieretoy}}
Our study deals with the following parabolic problem indexed by $M$:
\begin{equation}\label{Cauchytoymodel}
 \begin{cases}
\partial_t u(t,x) -\partial_{xx} u(t,x) = f(u)\lp 1+r(x)\rp \qquad t>0,x\in\R \\
u(0,x)=\vp(x) \\
r(x)=g(x-M).
 \end{cases}
\end{equation}
Therefore, the perturbation term satisfies
\begin{equation}
 \label{hypothesesonr}
 \lV r(x) \rV\leq e^{\kappa\lp x-M\rp},\qquad r>-\varpi.
\end{equation}
We have the following stability result.
\begin{prop}\label{toyprop} 
 Let $u$ be the solution of the Cauchy problem (\ref{Cauchytoymodel}). Under the assumptions (\ref{hypbistable}) and (\ref{hypothesesonr}) on $f$ and $r,$ 
 there exists $\gamma>0,$ there exist $M_0>0,K>0,N_0>0$ such that, for all $M\geq M_0$, for all $\displaystyle t\in\lp 0,\frac{M}{c}-N_0\rp,$ for all $x\in\R,$
 $$
 \lb u(t,x)-\vp(x-ct)\rb \leq K e^{\gamma \lp ct-M\rp}.
 $$
\end{prop}

\subsection{Proof of Theorem \ref{solentieretoy} with Proposition \ref{toyprop}}

The proof of Theorem \ref{solentieretoy} is quite classical, see \cite{BHM} for instance. For all integer $n,$ let $u_n,$ defined
for $t\geq -n$ and $x\in\R,$ be the solution of the following parabolic Cauchy problem:
\begin{equation}
 \begin{cases}
\partial_t u_n(t,x) -\partial_{xx} u_n(t,x) = f(u_n)\lp 1+g(x)\rp \qquad t>-n,x\in\R \\
u_n(-n,x) = \vp(x+cn).
 \end{cases}
\end{equation}
From classical parabolic estimates, $\lp u_n\rp_n$ converges up to an extraction locally uniformly to some function $u_\infty$ defined for $t\in\R,x\in\R.$
Then we apply Theorem \ref{toyprop}: 
\begin{equation}
 \forall n>\frac{M_0}{c},\ \forall t<-N_0,\ \lV u_n(t,.)-\vp(.-ct)\rV_{L^\infty(\R)}\leq K e^{\gamma ct},
\end{equation}
and the proof of Theorem \ref{solentieretoy} is concluded by letting $n$ tend to $+\infty.$

\subsection{Proof of Proposition \ref{toyprop}: splitting of the problem}\label{splittingpetit}
It is natural to consider the equation 
in the moving framework. Hence we make the following change of variable
\begin{equation*}
\xi =  x-ct, \qquad \ut(t,\xi)=u(t,x)=\ut(t,x-ct).
\end{equation*}
The problem under study becomes
\begin{equation}\label{CauchytoymodelTW}
 \begin{cases}
\partial_t \ut(t,\xi) -c\partial_\xi \ut(t,\xi) -\partial_{\xi\xi} \ut(t,\xi) = f(\ut)\lp 1+r(\xi+ct)\rp \qquad t>0,\xi\in\R \\
\ut(0,\xi)=\vp(\xi).
 \end{cases}
\end{equation}
Using the decomposition (\ref{decompositionimagenoyau}) we have the next lemma.

\begin{lem}\label{lemnoyauimage}
Let $u$ be the solution of the Cauchy problem (\ref{CauchytoymodelTW}). There exists $\e_1>0$ such that if
$$T_{max}=\sup\{ T\geq0,\ \forall t\in[0,T],\ \lV \ut(t)-\vp\rV_\infty\leq \e_1\},$$ 
there exist two functions $\chi\in C\lp[0,T_{max})\rp$ and 
$v\in C\lp[0,T_{max}),X\rp$ such that, for all $t\in (0,T_{max}),$
\begin{equation}\label{decompositionhenry}
\ut(t,\xi)=\vp(\xi+\chi(t))+v(t,\xi) 
\end{equation}
where $v$ satisfies, for all $t\in[0,T_{max}),$
\begin{equation*}
 \langle e_*,v(t)\rangle=0.
\end{equation*}
\end{lem}

The proof is a consequence of the implicit function theorem. See \cite{henry} for instance for a guideline of the proof.
Such a decomposition is common when dealing with stability of travelling waves and comes from the invariance under translation in space
of (\ref{bistable}).

Inserting the ansatz (\ref{decompositionhenry}) in our problem (\ref{CauchytoymodelTW}) yields the following equation:
\begin{equation}\label{decomposition1}
 \chi'\vp'(\xi+\chi)+v_t-c\vp'(\xi+\chi)-cv_\xi-\vp''(\xi+\chi)-v_{\xi\xi}=f\lp\vp(\xi+\chi)+v\rp\lp1+r(\xi+ct)\rp
\end{equation}
We make a Taylor expansion for the terms $\chi'\vp'(\xi+\chi)$ and $f\lp\vp(\xi+\chi)+v\rp.$ From now, for the sake of 
simplicity and when there is no possible confusion, we will sometimes omit te variables. 
In particular, we will use the
following notations:
\begin{equation}\label{notationsxiRr}
  \vp:=\vp(\xi) \qquad \vp_\chi:=\vp(\xi+\chi(t)) \qquad
 r:=r(\xi+ct). 
\end{equation}
Equation (\ref{decomposition1}) becomes
\begin{equation*}
 \chi'\vp'+v_t-cv_\xi-v_{\xi\xi}-f'(\vp)v-r(\xi+ct)f'(\vp)v-r(\xi+ct)f(\vp_\chi)=R(t,\xi,\chi,v)
\end{equation*}
where the right term is given by
\begin{equation}\label{prereste}
 R(t,\xi,\chi,v)= v\chi\vp'(b_1)f''(b_3)-\chi\chi'\vp''(b_2)+v^2f''(b_4)(1+r)
\end{equation}
with 
\begin{equation*}
 b_1,b_2 \;\in \lp \xi-|\chi|,\xi+|\chi|\rp, \qquad b_3,b_4\in \lp -\lV\vp\rV_\infty-|v|,\lV\vp\rV_\infty+|v|\rp.
\end{equation*}
We write this term in a more convenient form. As soon as $\chi<1,$ there exist $\Phi_1,\Phi_2,\Phi_3$ uniformly bounded functions of $(t,\xi)$ 
with bounds depending only on $\lV f\rV_{C^2},$ $\lV\vp\rV_{C^2},$ and $\lV r\rV_\infty$ such that 
\begin{equation}\label{reste}
 R(t,\xi,\chi,v)=v\chi\vp'\Phi_1+\chi\chi'\lp \vp'+f(\vp)\rp\Phi_2+v^2\Phi_3.
\end{equation}
Hence, using the decomposition (\ref{decompositionimagenoyau}), the problem (\ref{CauchytoymodelTW}) is equivalent to the following 
system on $[0,T_{max}):$
\begin{equation}\label{systemedecompose}
\begin{cases}
 \chi'-\langle e_*,r(.+ct)\lp f'(\vp)v+f(\vp_\chi)\rp\rangle = \langle e_*,R(t,.,\chi,v)\rangle \qquad t\in[0,T_{max})\\
 v_t-\ML v-Q\lc r(.+ct)\lp f'(\vp)v+f(\vp_\chi)\rp \rc = 
 Q\lc R(.,\chi,v)\rc \quad t\in[0,T_{max}),\xi\in\R \\
 \chi(0)=0,\ v(0,\xi)=0.
\end{cases}
\end{equation}

For all $\a>0,$ for all $M>0,$ let us define $T(\a,M)$ by
\begin{equation}\label{talphabeta}
T(\a,M):=\sup\{ T>0,\ \forall t<T,\ \max\{|\chi(t)|,|\chi'(t)|,\lV v(t)\rV_\infty\}\leq e^{\a\lp ct- M\rp}\}. 
\end{equation}

\subsection{Equation on $\MR(\ML)$}
\begin{lem}\label{lemmeimage}
 Let $\chi,v$ be solution of (\ref{systemedecompose}). There exist $\a,\gamma>0$ with $\gamma>\a,$ there exist $M_1>0$ and $N_1>0$ and a constant
 $C_3>0$ 
 such that for all $M>M_1,$ for all $\displaystyle t<\inf \{T(\a,M),\frac{M}{c}-N_1,T_{max}\},$ 
 $$
 \lV v(t)\rV_{L^\infty(\R)}\leq C_1 e^{\gamma\lp ct- M\rp}.
 $$
\end{lem}

\paragraph{Proof of Lemma \ref{lemmeimage}}
The equation under study is the following:
\begin{equation}\label{equationimage}
\begin{cases}
 v_t-\ML v-\lp rf'(\vp)v-\langle e_*,rf'(\vp)v\rangle\vp' \rp = Q\lc R+rf(\vp_\chi)\rc \\
 v(0,\xi)=0.
 \end{cases}
\end{equation}

\subparagraph{Symmetrisation of the problem} Let $v$ be the solution of (\ref{equationimage}) on $[0,T_{max}).$ We define
for $t\in[0,T_{max})$ and $\xi\in\R$
\begin{equation}\label{defw}
 w(t,\xi):=e^{\frac{c}{2}\xi}v(t,\xi).
\end{equation}
%
%
The function $w$ satisfies the following equation:
\begin{equation}\label{equationsurw}
 w_t-w_{\xi\xi}+\lp \frac{c^2}{4}-f'(\vp)\rp w-rf'(\vp)w = 
 -\langle e_*,rf'(\vp)v\rangle\vp'e^{\frac{c}{2}\xi}+e^{\frac{c}{2}\xi}Q\lc R+rf(\vp_\chi)\rc.
\end{equation}

\subparagraph{Energy estimates on $w$} Multiply (\ref{equationsurw}) by $w$ and integrate by parts ; as $v$ is orthogonal with $\vp',$ 
it comes $\int_\R e^{c\xi}\vp'(\xi)v(t,\xi)d\xi=0$ for all $t<T_{max}.$ We get:
\begin{equation}\label{equationintegree}
 \frac{1}{2}\frac{d}{dt}\int_\R w^2+\int_\R w_\xi^2+\int_\R\lp\frac{c^2}{4}-f'(\vp)\rp w^2-\int_\R rf'(\vp)w^2=
 \int_\R e^{c\xi}vQ\lc R+rf(\vp_\chi)\rc.
\end{equation}
Since $v=e^{-\frac{c}{2}\xi}w$ lies in the orthogonal of $N(\ML),$ and knowing that 0 is an isolated eigenvalue in the spectrum of $\ML,$
we have for all $t\in[0,T_{max})$:
\begin{equation}\label{Alessandro}
 \int_\R w_\xi^2+\int_\R\lp\frac{c^2}{4}-f'(\vp)\rp w^2 \geq \rho_1 \int_\R w^2.
\end{equation}
Let us set $\zeta>0$ such that
\begin{equation}\label{sharpspectral}
 r>-\frac{\rho_1}{\lV f'\rV_\infty}+\frac{\zeta}{\lV f'\rV_\infty}.
\end{equation}
The function $f$ is $C^1$ with $f'(0)<0.$ Hence, there exists $X_1$ such that for all $\xi>X_1,$ $\displaystyle f'(\vp(\xi))<0.$
Hence we get for all $\xi\in\R,$ for all $\displaystyle t\leq\frac{1}{c}\lp M-X_1+\log\lp\frac{\rho_1-\zeta}{\lV f'\rV_\infty}\rp\rp,$
\begin{equation}\label{domination1}
r(\xi+ct)f'(\vp(\xi))\leq \rho_1-\zeta.
\end{equation}
Let us set
\begin{equation}\label{defN1}
 N_1=\max \left\{0, \frac{1}{c}\lp X_1-\log\lp\frac{\rho_1-\zeta}{\lV f'\rV_\infty}\rp\rp \right\}, \quad t\leq\inf \{T(\a,M),\frac{M}{c}-N_1,T_{max}\}.
\end{equation}
Using (\ref{Alessandro}) and (\ref{domination1}) in (\ref{equationintegree}) and a Young inequality yield:
\begin{align}
 \frac{1}{2}\frac{d}{dt}\lV w\rV_2^2
 +\zeta\lV w\rV_2^2 & \leq  
 \int_\R \lb e^{\frac{c}{2}\xi}w(t,\xi)r(\xi+ct)f(\vp_\chi)\rb d\xi+\int_\R \lb e^{\frac{c}{2}\xi}w(t,\xi)R(\xi,\chi,v) \rb d\xi \nonumber\\
 & \leq \frac{\zeta}{2}\lV w\rV_2^2+\frac{2}{\zeta}\int_\R e^{c\xi}r^2f^2(\vp_\chi)d\xi+\int_\R \lb e^{\frac{c}{2}\xi}wR \rb d\xi \label{domination2}.
\end{align}
The second term in (\ref{domination2}) satisfies:
\begin{align}
 \int_\R e^{c\xi}r^2(\xi+ct)f^2\lp\vp_\chi\rp  & \leq \lV f\rV_\infty^2\int_{-\infty}^{0}e^{c\xi}r^2(\xi+ct)d\xi+\int_0^{+\infty}r^2(\xi+ct)e^{c\xi}f^2(\vp_\chi)d\xi \label{domination3} \\
  & \leq  \frac{\lV f\rV_\infty^2}{2\kappa+c}e^{2\kappa \lp ct-M\rp} \nonumber \\
  & + e^{2\kappa(ct-M)}\int_0^{M-ct}e^{(2\kappa+c)\xi}f^2\lp\vp(\xi+\chi)\rp d\xi \nonumber \\
  &+ \lV r\rV_\infty\int_{M-ct}^{+\infty}e^{c\xi}f^2\lp\vp(\xi+\chi)\rp d\xi. \nonumber
\end{align}
We linearise the function $f$ in 0 and use (\ref{estimeesbistable}). There exists a constant $C_4$ such that, as soon as $\lb\chi(t)\rb\leq 1$ and $\xi>0,$
\begin{equation}\label{fphixiinfty}
\lb f\lp \vp(\xi+\chi(t))\rp\rb \leq C_4 e^{\la \xi}.
\end{equation}
This gives
\begin{align}
 e^{2\kappa(ct-M)}\int_0^{M-ct}e^{(2\kappa+c)\xi}f^2\lp\vp(\xi+\chi)\rp d\xi & \leq e^{2\kappa(ct-M)}\frac{C_4^2}{2\kappa+c+2\la}\lp e^{\lp2\kappa+c+2\la\rp(M-ct)}-1\rp \nonumber \\
  & \leq e^{2\kappa(ct-M)}\frac{2C_4^2}{-2\kappa-c-2\la} \label{domination31}
\end{align}
and
\begin{equation}\label{domination32}
 \int_{M-ct}^{+\infty}e^{c\xi}f^2\lp\vp(\xi+\chi)\rp d\xi \leq \frac{C_4^2}{-c-2\la}e^{(-c-2\la)(ct-M)}.
\end{equation}
From (\ref{hypdomain4}), $-c-2\la>2\kappa.$ We insert (\ref{domination31}) and (\ref{domination32}) in (\ref{domination3}). There exists a constant
$K_1$ such that 
\begin{equation*}
  \int_\R e^{c\xi}r^2(\xi+ct)f^2\lp\vp_\chi\rp d\xi \leq K_1 e^{2\kappa(ct-M)}.
\end{equation*}
Let us go back to (\ref{domination2}). It yields:
\begin{align}
  \frac{1}{2}\frac{d}{dt}\lV w(t)\rV_2^2+\frac{\zeta}{2}\lV w(t)\rV_2^2  & \leq K_1 e^{2\kappa\lp ct- M\rp}+\int_\R \lb e^{\frac{c}{2}\xi}wR \rb d\xi \nonumber \\
  & \leq K_1 e^{2\kappa\lp ct- M\rp} +\frac{\zeta}{4}\lV w(t)\rV_2^2+\frac{4}{\zeta}\int_\R e^{c\xi}R^2(t,\xi) d\xi. \label{domination5}
\end{align}
 Recall that $R$ is defined by (\ref{reste}). 
 From (\ref{estimeesbistable}), the functions $\xi\mapsto e^{c\xi}\vp'^2(\xi)$ and $\xi\mapsto e^{c\xi}f^2(\vp(\xi))$ belong to $L^1(\R).$
 Hence, as soon as $t$ satisfies (\ref{defN1}), the first term in (\ref{domination5}) satisties for some constant $C$ 
 \begin{align}
  \int_\R e^{c\xi}R^2(t,\xi)d\xi & \leq C e^{4\a\lp ct-M\rp}+3\int_\R e^{c\xi}v^4(t,\xi)\Phi_3^2(t,\xi)d\xi \nonumber \\
  & \leq C e^{4\a\lp ct-M\rp}+3\lV\Phi_3\rV_\infty^2\int_\R w^2(t,\xi)v^2(t,\xi)d\xi \nonumber \\
  & \leq C e^{4\a\lp ct-M\rp}+3\lV\Phi_3\rV_\infty^2\lV w(t)\rV_2^2e^{-2\a N_1}. \label{domination6}
 \end{align}
 We set:
\begin{equation}\label{choixalpha}
 \a=\frac{\kappa}{2}.
\end{equation}
Under the assumptions of Lemma \ref{lemmeimage}, from (\ref{domination5}) and (\ref{domination6}), up to a greater $N_1$ we get that the function $w$ satisfies
\begin{equation}\label{estimationGronwall}
 \frac{1}{2}\frac{d}{dt}\lV w(t)\rV_2^2+\frac{\zeta}{8}\lV w(t)\rV_2^2  \leq e^{4\a\lp ct- M\rp}\lp K_1+C\rp.
\end{equation}
With a Gronwall argument in (\ref{estimationGronwall}), we get that for $M$ large enough, for some constant $K_2$ which does not depend on $M,$ 
\begin{equation}\label{dominationl2}
 \lV w(t)\rV_{L^2} \leq K_2e^{2\a(ct- M)},\qquad  0\leq t<\inf\{ T(\a,M),T_{max},\frac{M}{c}-N_1\}.
\end{equation}

\subparagraph{Parabolic estimates on $w$} From the energy estimates on $w$ we derive on this paragraph $L^\infty$ estimates. From (\ref{equationsurw}), the function $w$ 
satisfies a parabolic equation of the form
\begin{equation}\label{parabolic}
 w_t-w_{\xi\xi}+a(t,\xi)w=g(t,\xi)
\end{equation}
with
\begin{equation}\label{aetg}
 \begin{cases}
a(t,\xi) = \frac{c^2}{4}-f'(\vp(\xi))\lp1+r(\xi+ct)\rp -v(t,\xi)\Phi_3(t,\xi)\\
g(t,\xi) =-\langle e_*,rf'(\vp)v\rangle\vp'(\xi)e^{\frac{c}{2}\xi}+e^{\frac{c}{2}\xi}Q\lc R+rf(\vp_\chi)\rc-v(t,\xi)\Phi_3(t,\xi)w(t,\xi).
 \end{cases}
\end{equation}
Under the assumptions of Lemma \ref{lemmeimage} the function $a$ clearly belongs to $\displaystyle L^\infty\lp[0,T]\times\R\rp$ uniformly in $T$
for all $T<\inf\{ T(\a,M),T_{max},\frac{M}{c}-N_1\}.$ 
Let us prove that the funtion $g$ also satisfies this property.
The function $a$ clearly belongs to $L^\infty\lp [0,+\infty[\times\R\rp.$ Let us show that the function $g$ belongs to $L^\infty\lp [0,T]\times\R\rp,$ uniformly in $T$ 
From (\ref{estimeesbistable}) and the smoothness of $f$ the functions
$x\mapsto e^{\frac{c}{2}\xi}\vp'(\xi)$ and $\xi\mapsto e^{\frac{c}{2}\xi}f(\vp(\xi+\chi))$ are uniformly bounded provided that $\chi$ remains bounded.
There exists $K_3>0,$ 
$$
\lb e^{\frac{c}{2}\xi}f(\vp(\xi+\chi))\rb+ \lb e^{\frac{c}{2}\xi}\vp'(\xi)\rb \leq K_3,\qquad \forall \xi\in\R.
$$
The function $g$ is given by 
$$
g(t,\xi) = e^{\frac{c}{2}\xi}\lp \vp'\chi v\Phi_1+\chi\chi'\lp \vp'+f(\vp)\rp\Phi_2 +rf(\vp_\chi)-\langle e_*,R+rf(\vp_\chi)+rvf'(\vp)\rangle\rp.
$$
Hence we have
\begin{align}
 \lb g(t,\xi)\rb \leq & ~K_3 \lp \lV\Phi_1\rV_\infty \lb \chi v\rb+\lV\Phi_2\rV_\infty (1+\lV f'\rV_\infty)\lb \chi\chi'\rb+\lV r\rV_\infty\lV f\rV_\infty\rp \nonumber \\
  & +K_3\lp\lV e_*\rV\lp |R|+\lV r\rV_\infty\lV f\rV_{C^1}(1+|v|)\rp\rp. \nonumber
\end{align}
This inequality provides the desired $L^\infty$ estimate provided that $v,\chi,\chi'$ remain bounded. Now we can apply classical parabolic estimates (see \cite{Lady}, Thm 8.1 p.192 for instance)
for (\ref{parabolic}) with our previous estimate (\ref{dominationl2}).
There exists a constant $K_4$ such that, for all $T<\inf\{ T(\a,M),T_{max},\frac{M}{c}-N_1\},$
\begin{equation}\label{estimeelinfiniw}
 \lV w\rV_{L^\infty\lp [0,T]\times \R\rp}\leq K_4e^{2\a\lp cT-M\rp}.
\end{equation}

\subparagraph{From $w$ to $v$} The function $v$ is given by $v=e^{-\frac{c}{2}\xi}w.$ Let us set 
\begin{equation}\label{definitionA}
 \sigma=\frac{\kappa}{2c},\quad T<\inf\{ T(\a,M),T_{max},\frac{M}{c}-N_1\},\quad A(T)=\sigma\lp M-cT\rp>0.
\end{equation}
The function $v$ satisfies for $t\in[0,T]$ on the half line $(-\infty,-A(T)]$
\begin{equation}\label{equationsurvA}
 \begin{cases}
v_t-cv_\xi-v_{\xi\xi}-f'(\vp)(1+r)v=\langle e_*,rf'(\vp)v\rangle \vp'+Q\lc R+rf(\vp_\chi)\rc \\
v(0,\xi)=0,\ v(t,-\infty)=0,\ v(t,-A)=w(t,-A)e^{\frac{c}{2}A(T)}.
 \end{cases}
\end{equation}
We have that $f'(\vp(\xi))\to f'(1)<0$ as $\xi\to-\infty.$ Hence, up to a larger $N_1$, 
$$
-f'(\vp)(1+r(\xi+ct))\geq -\frac{1}{2}f'(1)>0,\qquad \forall t\in[0,T],\forall \xi\in(-\infty,-A(T)].
$$
We recall that $e_*$ is defined by (\ref{definitioneetoile}) and is continuous, and so is $Q.$ The function $\vp'$ satisfies
(\ref{estimeesbistable}). The perturbative term satisfies (\ref{hypothesesonr}). Hence, for some constant $K_5,$
for all $t\in[0,T],$ for all $\xi\in(-\infty,-A(T)],$ considering hypotheses of Lemma \ref{lemmeimage},
\begin{align}
 \lb \langle e_*,rf'(\vp)v\rangle \vp'+Q\lc R+rf(\vp_\chi)\rc\rb & \leq K_5\lp \lV v\rV_\infty e^{-\mu A}+\lV v\rV_\infty^2+\chi^2+\lb \chi\chi'\rb+e^{\kappa(-A+cT-M)}\rp \nonumber \\
 & \leq K_5\lp e^{(\a+\sigma\mu)(cT-M)}+3e^{2\a(cT-M)}+e^{{2\a(1+\sigma)(cT-M)}}\rp. \label{termesourcevA}
\end{align}
Indeed, we have fixed $2\a=\kappa$ in (\ref{choixalpha}). Finally, using (\ref{estimeelinfiniw}), we have
\begin{equation}\label{consitionaubordvA}
|v(t,A(T))|\leq K_4 e^{(2\a-\frac{\sigma c}{2})(cT-M)}.
\end{equation}
Combining (\ref{termesourcevA}) and (\ref{consitionaubordvA}) in (\ref{equationsurvA}) with the parabolic maximum principle and
our previous estimate (\ref{estimeelinfiniw}) for $\xi>-A,$ we get, for some constant $C_3$:
\begin{equation}
 \lV v\rV_{L^\infty\lp[0,T]\times\R\rp}\leq C_3e^{\gamma(cT-M)},\ \textrm{ with }\gamma=\inf\{ 2\a-\frac{3\a}{2},2\a(1+\sigma),\a+\sigma\mu\}
\end{equation}
which concludes the proof of Lemma \ref{lemmeimage}.
\qed

\subsection{Equation on $\MN(\ML)$}
\begin{lem}\label{lemmenoyau}
 Let $\chi,v$ be solution of (\ref{systemedecompose}). Let $\a=\frac{\kappa}{2},\gamma, M_1, N_1$ given by Lemma \ref{lemmeimage}. 
There exists a constant $C_4$ such that for all $M>M_1,$ for all $t<\inf\{ T(\a,M),T_{max},\frac{M}{c}-N_1\},$
$$
\max\{|\chi(t)|,|\chi'(t)|\} \leq C_4 e^{\gamma(ct- M)}.
$$
\end{lem}

\begin{proof}
 The shift function $\chi$ satisfies the following equation, still using the notations given by (\ref{notationsxiRr}):
 \begin{equation}\label{equationxi}
  \begin{cases}
 \chi'(t) = \langle e_*, R\rangle + \langle e_*, r(f'(\vp)v+f(\vp_\chi))\rangle \\
 \chi(0) = 0.
  \end{cases}
 \end{equation}
Hence
\begin{equation}\label{inegalitexipoint}
 \lb \chi'(t)\rb \leq \lb \langle e_*,R\rangle\rb + \lb \langle e_*,rf'(\vp)v\rangle\rb + \lb \langle e_*,rf(\vp_\chi)\rangle\rb.
\end{equation}
The higher order term $R$ is given by (\ref{reste}). Under the assumptions on $T(\a,M)$ it yields
\begin{align}
\lb \langle e_*,R\rangle\rb & \leq  \lV e_*\rV.\lV R\rV_\infty \nonumber \\
 & \leq \lV e_*\rV e^{2\a(ct-M)}\lp C_2(\lV\Phi_1\rV_\infty+\lV\Phi_2\rV_\infty)+\lV\Phi_3\rV\rp_\infty. \label{inegalitexipoint1}
\end{align}
From Lemma \ref{lemmeimage} we get 
\begin{equation}\label{inegalitexipoint2}
 \lb \langle e_*,rf'(\vp)v\rangle\rb \leq C_3 \lV e_*\rV.\lV r\rV_\infty\lV f'\rV_\infty e^{\gamma(ct- M)}.
\end{equation}
Let us deal with the last term. We have:
\begin{align}
 \lb \langle e_*,rf(\vp_\chi)\rangle\rb & = \lb \int_\R e^{c\xi}\vp'(\xi)r(\xi+ct)f(\vp(\xi+\chi))d\xi\rb \label{inegalitexipoint3} \\
  & \leq \lV \vp'\rV_\infty \lV f\rV_\infty \int_{-\infty}^0 e^{c\xi}e^{\kappa(\xi+ct-M)}d\xi \nonumber \\
  & + \int_0^{M-ct}\lb e^{c\xi}e^{\kappa(\xi+ct-M)}\vp'(\xi)f(\vp(\xi+\chi))\rb d\xi \nonumber \\
  & + \lV r\rV_\infty\int_{M-ct}^{+\infty}\lb e^{c\xi} \vp'(\xi)f(\vp(\xi+\chi))\rb d\xi. \nonumber
\end{align}
Once again, we use (\ref{estimeesbistable}) and (\ref{fphixiinfty}). It yields
\begin{equation}\label{inegalitexipoint31}
 \int_0^{M-ct}\lb e^{c\xi}e^{\kappa(\xi+ct-M)}\vp'(\xi)f(\vp(\xi+\chi))\rb d\xi \leq \frac{C_2C_4}{-c-2\la}e^{\kappa(ct-M)}
\end{equation}
and
\begin{equation}\label{inegalitexipoint32}
 \int_{M-ct}^{+\infty}\lb e^{c\xi} \vp'(\xi)f(\vp(\xi+\chi))\rb d\xi \leq \frac{C_2C_4}{-c-2\la}e^{(-c-2\la)(ct-M)}.
\end{equation}
We have that $\kappa\leq -c-2\la$ by hypothesis (\ref{hypdomain4}). Hence, using (\ref{inegalitexipoint31}) and (\ref{inegalitexipoint32}) in (\ref{inegalitexipoint3})
gives that, for some constant $K_5$ and under hypotheses of the lemma,
\begin{equation}\label{inegalitexipoint4}
   \lb \langle e_*,rf(\vp_\chi)\rangle\rb \leq K_5 e^{\kappa(ct-M)}.
\end{equation}
Hence, (\ref{inegalitexipoint1}), (\ref{inegalitexipoint2}) and (\ref{inegalitexipoint4}) in (\ref{inegalitexipoint}) gives, for some constant $K_6,$
\begin{equation}\label{inegalitexipointfin}
 \lb \chi'(t)\rb \leq K_6 e^{\gamma(ct- M)}.
\end{equation}
It remains to integrate (\ref{inegalitexipointfin}) to get the desired result.
\end{proof}

\subsection{Conclusion of the proof}
Let $\a=\frac{\kappa}{2},\gamma,M_1,N_1$ given by Lemma \ref{lemmeimage}. We set 
\begin{equation}
 N_2=\frac{1}{c(\gamma-\a)}\log\lp\max \{1,C_3,C_4\}\rp.
\end{equation}
From Lemma \ref{lemmeimage} and Lemma \ref{lemmenoyau}, we get that, for all $M>\max\{M_1,cN_2\},$
$$
T(\a,M)\geq \min\left\{ T_{max},\frac{M}{c}-N_2,\frac{M}{c}-N_1\right\}.
$$
Now we deal with $T_{max}$ given by Lemma \ref{lemnoyauimage}. We have, for all $t,M$ satisfying hypotheses of Lemma \ref{lemmeimage},
\begin{align}
 \lV u(t)-\vp\rV_\infty & = \lV \vp(.+\chi(t))+v(t)-\vp\rV_\infty \nonumber \\
  & \leq  \lV \vp(.+\chi(t))-\vp\rV_\infty + \lV v(t)\rV_\infty \nonumber \\
  & \leq \lb\chi(t)\rb\lV\vp'\rV_\infty+\lV v(t)\rV_\infty \nonumber \\
  & \leq \lp\lV\vp'\rV_\infty+1\rp\max\{C_3,C_4\} e^{\gamma(ct-M)}. \label{finpertubation}
\end{align}
Let us set 
$$
N_3 = \frac{1}{\gamma c}\log\lp \frac{\lp\lV\vp'\rV_\infty+1\rp\max\{C_3,C_4\}}{\e_1}\rp,\ N_0=\max\{N_1,N_2,N_3\},\ M_0=\max \{M_1,cN_0\}
$$
and, according to Lemma \ref{lemnoyauimage}, for all $M>M_0,$ $T_{max}>\frac{M}{c}-N_0.$ The proof is then concluded by (\ref{finpertubation}).
\qed

\section{Cylinder-like domains: proof of Theorem \ref{thmgrosmodele}}
The proof of the existence part relies on the same kind of arguments that are used in the previous section.
The main new argument is Lemma \ref{lemtail} which provides a control of the solutions ahead of the front, where heterogeneities are no longer negligible.
The point is again a stability result. We consider the following problem, indexed by $M:$ 
 \begin{equation}\label{cauchygros}
 \begin{cases}
\partial_t u- \Delta u = f(u) & t>0,\ (x,y)\in\Omega \\
\partial_\nu u=0 & t>0,\ (x,y)\in\partial\Omega \\
u(0,x,y)=\vp(x+M):=\tau_M\vp(x).
 \end{cases}
\end{equation}

Throughout this section, we will denote by $C$ a generic positive constant, which may differ from place to place 
even in the same chain of inequalities. Moreover, we will use classical notations concerning Hölder spaces:
\begin{itemize}
 \item For any $\d\in(0,1)$ and integer $k,$ if $\Omega_0$ is a spatial domain included in $\R^{N}$ or $\R{N+1},$ $C^{k,\d}(\Omega_0)=C^{k+\d}$ is the space of functions whose derivatives 
 up to order $k$ lie in $C^\d(\Omega_0),$ the space of $\d-$Hölder functions.
 \item For parabolic problems with functions depending on the time $t$ and the space, the space $C^{k,\d}\lp [T^-,T^+]\times\Omega_0\rp$ is the space 
 of functions that are $C^{\frac{k}{2}+\frac{\d}{2}}$ in time and $C^{k,\d}$ in space.
\end{itemize}
The main ingredient is the following. 
\begin{prop}\label{grosprop}
 Under assumptions (\ref{hypbistable}) on $f$ and (\ref{hypdomain1}) and (\ref{hypdomain2}) on the domain $\Omega,$ there exist $\gamma>0$ 
 and two positive constants $M_0,N_0$ such that for all $M>M_0,$ for all $\displaystyle T\in\lp 0,\frac{M}{c}-N_0\rp,$ for all $(x,y)\in\Omega,$ the solution $u$ of 
 (\ref{cauchygros}) satisfies
 \begin{equation}
 \lV u(t,x,y)-\tau_M\vp(x-ct)\rV_{C^{2,\d}\lp [0,T]\times\Omega\rp} \leq C e^{\gamma\lp cT-M\rp}.
 \end{equation}
\end{prop}

In the following subsection, we prove our main result Theorem \ref{thmgrosmodele} using the above proposition.
The rest of the paper is devoted to the proof of this proposition.

\subsection{Proof of Theorem \ref{thmgrosmodele} using Proposition \ref{grosprop}}
The argument is the same as for the proof of theorem \ref{solentieretoy} and comes from \cite{BHM}. For all integer $n$ 
let $u_n$ be the solution of the Cauchy problem
\begin{equation}\label{Cauchysuite}
 \begin{cases}
\partial_t u_n-\Delta u_n = f(u_n),\ t>-n,\ (x,y)\in\Omega \\
\partial_\nu u_n = 0, \ t>-n,\ (x,y)\in\partial\Omega \\
u(-n,x,y)=\vp(x+cn).
 \end{cases}
\end{equation}
From parabolic estimates the sequence $\lp u_n\rp_n$ converges up to extraction to some entire function $u_\infty$ locally uniformly 
together with its second derivative in space and first in time. Moreover, from proposition \ref{grosprop}, passing to the limit we get, 
for all $T<N_0$
\begin{equation}\label{convforte}
 \lV u(t,x,y)-\vp(x-ct)\rV_{C^2(\Omega)\times C^1(-\infty,T)} \leq Ce{\gamma cT}
\end{equation}
which gives the desired property (\ref{convergencetowave}).

It remains to prove the uniqueness of the entire solution. The  main point is to get a bouond from belo for the time derivative of $u_\infty.$
Following \cite{BHM}, we define for all $\eta$ small enough 
$$
\Omega_\eta(t)=\left\{ (x,y)\in\Omega, \eta\leq u_\infty(t,x,y)\leq 1-\eta\right\}.
$$
We have the following lemma:
\begin{lem}
For any $\eta\in\lp0,\frac{1}{2}\rp$ there exist $\d>0$ and $T(\eta,\d)\in\R$ such that
$$
\partial_t u_\infty (t,x,y)\geq \d
$$
for all $t\in (-\infty,T(\eta,\d)],$ for all $(x,y)\in\Omega_\eta(t).$
\end{lem}

The proof is immediate and comes from the convergence of the time derivative in (\ref{convforte}) and the fact that the wave $\vp$ 
is decreasing. The uniqueness thus follows from \cite{BHM}, section 3.
\qed

\subsection{A preliminary result: travelling super-solutions}
To control the tail of our solution when the domain is heterogeneous, we will use the following lemma.
\begin{lem}\label{lemtail}
 There exist $\a_1,\a_2>0$ and a positive function $\psi$ with $\a_2\leq \psi\leq\frac{1}{\a_2}$ such that for all $\a\leq \a_1,$ for all $\e>0,$
 $$
 \us : (t,x,y)\mapsto \e \psi(x,y)e^{-\a (x-ct)}
 $$
 is a super-solution of the problem (\ref{Neumannpb}) on the set $\displaystyle \Omega \cap \left\{ x>ct\right\}.$
\end{lem}

\begin{proof}
 Let us recall that a super-solution in a domain with boundary such as (\ref{Neumannpb}) is a function $\us$ which satisfies the following inequality
 \begin{equation}\label{supergeneral}
  \begin{cases}
 \partial_t \us -\Delta \us -f(\us) \geq 0 & t\geq0,(x,y)\in\Omega \\
 \partial_\nu \us \geq 0 & t\geq0,(x,y)\in\partial\Omega
  \end{cases}
 \end{equation}

\paragraph{Distance function to the boundary}
From our hypotheses (\ref{hypdomain1}), there exists $r>0$ such that $\Omega$ satisfies a sliding sphere 
property of radius $3r:$ for all $(x_0,y_0)\in\partial\Omega,$ there exists $(x_1,y_1)\in\Omega$ such that
\begin{equation}\label{slidingsphere}
  B\lp(x_1,y_1),3r\rp\subset\Omega\textrm{  and  }\overline{B\lp(x_1,y_1),3r\rp}\cap\partial\Omega=\{(x_0,y_0)\}.
\end{equation}
For all $(x,y)\in](x_0,y_0),(x_1,y_1)]$ we denote $(x_0,y_0):=\Pi(x,y)$ the orthogonal projection of $(x,y)$ on $\partial\Omega.$
It satisfies $dist\lp (x,y),\partial \Omega\rp = \lV (x,y)-\Pi(x,y)\rV.$ Let us set 
$\Omega_{2r}:=\left\{ (x,y)\in\Omega,\ dist\lp (x,y),\partial\Omega\rp\leq 2r\right\}$. It is well known (see \cite{GT}, section 14.6)
that the distance function to the boundary defined by
\begin{equation}\label{distancefct}
 d : \left\{ \begin{array}{ccl}
              \Omega_{2r} & \longrightarrow & \R^+ \\
              (x,y) & \longmapsto & dist\lp (x,y),\partial\Omega\rp
             \end{array}
 \right.
\end{equation}
is uniformly $C^2$ on $\Omega_{2r}$ and satisfies $\nabla d(x,y) = -\overrightarrow{\nu}\lp \Pi(x,y)\rp.$
 
\paragraph{Construction of $\psi$}
We need a cut-off function: let us choose some function $g$ of the form 
\begin{equation*}
g : \left\{
 \begin{array}{ccl}
  \R^+ & \longrightarrow & \R \\
  s & \longmapsto & \begin{cases}
                             s \textrm{ if }s\in[0,r] \\
                             \frac{3r}{2}\textrm{ if }s\geq 2r
                            \end{cases}
 \end{array}\right.
\end{equation*}
with $g$ smooth, monotonous, and $|g'|,|g''|<1.$
Now, for a positive constant $a$ to be chosen later, we define the function $\psi$ by
\begin{equation}\label{defpsi}
 \psi : \left\{  \begin{array}{ccl}
                  \Omega & \longrightarrow & \R \\
                  (x,y) & \longmapsto & \left\{ \begin{array}{lc}
                                                 \lp\frac{3r}{2}-g\lp d(x,y)\rp\rp+a & \textrm{ if } (x,y)\in\Omega_{2r} \\
                                                 a & \textrm{ if } (x,y)\in\Omega\setminus\Omega_{2r}.
                                                \end{array} \right.
                                                \end{array}  \right.
\end{equation}
This function $\psi$ is uniformly $C^2$ on $\Omega$ with $a\leq \psi\leq a+\frac{3r}{2}$ and $|\Delta\psi|,|\nabla\psi|\leq \lV d\rV_{C^2}.$
Moreover, on $\Omega_r,$ we have $\nabla \psi(x,y) = \overrightarrow{\nu}\lp \Pi(x,y)\rp,$ the outward normal derivative on $\partial\Omega.$

\paragraph{The super-solution}
Let us fix $\d>0$ small enough. For $\displaystyle \a\in\lp0, \frac{c+\sqrt{c^2-4(\fp+\d)}}{2}\rp$ we define 
$\displaystyle \o(\a)=\sqrt{\a c-\a^2-\fp-\d}.$ Let $\a_1$ be such that $\displaystyle\omega(\a_1)^2>-\frac{\fp}{2}.$ 
Now, let us choose $a>0,\a\in(0,\a_1)$ such that
\begin{equation*}
 a>-\frac{4\lV d\rV_{C^2}}{\fp},\qquad \a<\frac{1}{a+\frac{3r}{2}}.
\end{equation*}
The function $\psi$ defined in (\ref{defpsi}) satisfies 
\begin{equation*}
 \begin{cases}
 -\Delta\psi+2\a\partial_x \psi+\omega(\a)^2\psi\geq0, & (x,y)\in\Omega \\
 -\a\psi+1\geq0, & (x,y) \in\partial\Omega.
 \end{cases}
\end{equation*}
From the regularity of $f$ there exists $\e_1$ such that for all $s\in(0,\e_1\lV\psi\rV_\infty)$ we have $f(s)\leq (\fp+\d)s,$ 
and the function
\begin{equation*}
  \us : (t,x,y)\mapsto \e_1\psi(x,y) e^{-\a (x-ct)}
\end{equation*}
is a supersolution in the desired domain. The result follows easily.
\end{proof}

\subsection{From $\Omega$ to a straight cylinder}  
We use the change of variables provided by the diffeomorphism $\overrightarrow{\Phi}$ given in (\ref{hypdomain1}) which sends $\Omega$ into $\Omega^\infty$. 
With an abuse of notation 
we write $z=\Phi(x,y)=\lp \Phi_1(x,y), ... , \Phi_N(x,y)\rp$ and $y=\Phi^{-1}(x,z).$
In the same manner, if $(x,y)$ is in $\partial\Omega,$ we denote $\overrightarrow{\nu}^\infty(x,z)$ the outward normal derivative of $\Omega^\infty$ at a 
point $(x,z)=\overrightarrow{\Phi}(x,y)$ of $\partial\Omega^\infty.$ 
THe function $\Phi$ is thus uniformly $C^3$ on $\Omega$ and from (\ref{hypdomain2}) we have for all $X\in\R$ for some constant $C$
\begin{align}
 \lV D_x\Phi\rV_{C^1\lp\Omega\cap\{x<X\}\rp}+ \lV D_y\Phi-Id_y\rV_{C^1\lp\Omega\cap\{x<X\}\rp} + &  \nonumber \\ 
 \lV \overrightarrow{\nu}-\overrightarrow{\nu}^\infty\circ\overrightarrow{\Phi}\rV_{L^\infty\lp \partial\Omega\cap\{x<X\}\rp} & \leq C\min\left\{ e^{\k X},1\right\}. \label{dominationPhi}
\end{align}
Let us combine the above change of variables $z=\Phi(x,y)$ with the transformation to a moving framework $\xi=x-ct$ and set 
\begin{equation}\label{uetutilda}
 \ut(t,\xi,z)=u(t,x,y)=\ut(t,x-ct,\Phi(x,y))=\ut(t,\xi,\phi(\xi+ct,y))
\end{equation}
where $u$ is the solution of (\ref{cauchygros}).
The equation for $\ut$ in the $(t,\xi,z)-$variables is then given by
\begin{equation}\label{grosapreschange}
 \begin{cases}
\ut_t-c\ut_\xi-\ut_{\xi\xi}-\Delta_z\ut-f(\ut) = R_1(t,\xi,z) & t>0, (\xi,z)\in\Omega^\infty\\
\partial_{\nu^\infty}\ut = R_2(t,\xi,z) & t>0,(\xi,z)\in\partial\Omega^\infty \\
\ut(0,\xi,z)=\tau_M\vp(\xi)
 \end{cases}
\end{equation}
where the two residual terms are as follows:
\begin{align}
 R_1(t,\xi,z) = & \sum_{i=1}^N \lp \ut_{z_i}\frac{\partial^2\Phi_i}{\partial x^2}+2\ut_{\xi z_i}\frac{\partial \Phi_i}{\partial x}+
 \frac{\partial \Phi_i}{\partial x} \sum_{j=1}^N \ut_{z_i z_j}\frac{\partial \Phi_j}{\partial x}\rp + \sum_{i,j=1}^N \ut_{z_j}\frac{\partial^2 \Phi_j}{\partial y_i^2} \label{reste1} \\
  & + \sum_{i=1}^N \lp \ut_{z_i^2}\lp\lp\frac{\partial \Phi_i}{\partial y_i}\rp^2-1\rp+
  \sum_{\substack{j,k=1 \\ (j,k)\neq (i,i)}}^N \ut_{z_j z_k}\frac{\partial \Phi_j}{\partial y_i}\frac{\partial \Phi_k}{\partial y_i}\rp \nonumber  
\end{align}

\begin{align}
 R_2(t,\xi,z) = & \ \nu_x \sum_{i=1}^N \ut_{z_i}\frac{\partial \Phi_i}{\partial x} + \sum_{j=1}^N \nu_{y_j}\lp \ut_{z_j}\lp 1-\frac{\partial \Phi_j}{\partial y_j}\rp
 +\sum_{\substack{k=1 \\ k\neq j}}^N \ut_{z_k}\frac{\partial \Phi_k}{\partial y_j}\rp \label{reste2} \\
 & +\nabla \ut . \lp \overrightarrow{\nu}(\xi+ct,y)-\overrightarrow{\nu}^\infty(\xi+ct,z)\rp \nonumber
\end{align}
where $\Phi$ and all its derivatives have to be considered in the $(x,y)$-variables:
\begin{equation}\label{Phitranslate}
 \Phi=\Phi(x,y)=\Phi\lp\xi+ct,\Phi^{-1}(\xi+ct,z)\rp.
\end{equation}

As all the coefficients in (\ref{grosapreschange}-\ref{reste2}) are bounded 
we can apply the same decomposition as that given by Lemma \ref{lemnoyauimage}:
\begin{lem}\label{lemnoyauimagegros}
 Let $\ut$ be the solution of the Cauchy problem (\ref{grosapreschange}). There exists $\e_1$ such that, if 
 $$
 T_{max} = \sup \left\{ T>0,\lb \ut(t,\xi,z)-\tau_M\vp(\xi)\rb\leq\e_1,\ \forall t<T,(\xi,z)\in\Omega^\infty\right\},
 $$
 there exist two functions $\chi\in C^{2,\d}\lp[0,T_{max}]\times\omega^\infty\rp$ and $v\in C^{2,\d}\lp[0,T_{max}]\times\Omega^\infty\rp$ such that for all
 $t<T_{max},$ 
 \begin{equation}\label{decomposition}
 \ut(t,\xi,z)=\tau_M\vp\lp \xi+\chi(t,z)\rp+v(t,\xi,z) 
 \end{equation}
 where $v$ satisfies $\langle \tau_M e_*,v(t,.,z)\rangle=0$ for all $t<T_{max}$ and $z\in\omega^\infty.$
\end{lem}

\begin{remark}\label{pasdeM}
The considered decomposition is similar to (\ref{decompositionimagenoyau}) with a translation of $M.$ In order not to overburden the notations, \textbf{we will 
omit to mention this translation} in the operators $\ML,$ $P$ and $e_*$ and in the functions $\vp$ and $\vp'$ throughout the remainder of this paper. 
We write $v\in \MR(\ML).$
\end{remark}

\begin{remark}\label{dzvonR}
 Considering the regularity of the solution, operators $\ML$ and $\partial_{z_i}$ commute, and so do $e_*$ and $\partial_{z_i}$. Hence, $\xi\mapsto\partial_{z_i} v(t,\xi,z)$
 also belongs to $\MR(\ML).$
 \end{remark}

 \subsection{splitting of the problem}
 We insert the ansatz provided by Lemma \ref{lemnoyauimagegros} in (\ref{grosapreschange}). As in the previous section we use that $\vp$ is a solution of (\ref{bistableTW})
and Taylor expansions.  It yields the following equation in $(v,\chi):$
\begin{equation}\label{eqavtproj}
 \begin{cases}
\vp'\chi_t-\vp'\Delta_z\chi +v_t-\ML v-\Delta_z v = R_3   & (\xi,z)\in\Omega^\infty \\
 \lp\vp'\nabla_z\chi+\nabla_z v \rp.\overrightarrow{\nu}_{\omega^\infty} = R_4 & (\xi,z)\in\partial\Omega^\infty \\
  \ut(0,\xi,z) = \vp(\xi)
 \end{cases}
\end{equation}
where 
\begin{subequations}\label{Reste}
 \begin{align}
 R_3 = & R_1+\frac{v^2}{2}f''(b_1)+\vp''\sum_{i=1}^N \chi^2_{z_i}+\lp\chi\Delta_z\chi-\chi\chi_t\rp\vp''(\xi+b_2)+v\chi\vp'(\xi+b_3)f''(b_4) \label{Reste3} \\
 R_4 = & R_2+\vp''(\xi+b_5)\nabla_z\chi.\overrightarrow{\nu}_{\omega^\infty}. \label{Reste4}
 \end{align}
\end{subequations}

Now, as for the one dimensional case, we set for some $\a>0,$ for all $M>0,$
\begin{equation}\label{TaMgros}
 T(\a,M):=\sup \left\{ T>0,\forall t<T,\lV\chi\rV_{C^{2,\d}\lp[0,t]\times(-1,1)\rp}+\lV v \rV_{C^{2,\d}\lp[0,t]\times\Omega^\infty\rp}\leq e^{\a(ct-M)}\right\}.
\end{equation}

\subsection{Estimates on $R_3,R_4$}
Let us define
\begin{equation}\label{defalpha}
 \a:=\frac{1}{4}\min \{ \a_1,\kappa\},\qquad \e_2 = \inf \psi
\end{equation}
where $\kappa$ is set in (\ref{dominationPhi}) and $\a_1,\psi$ are given in Lemma \ref{lemtail}.

\begin{lem}\label{controlrestes}
Let $u$ be the solution of (\ref{cauchygros}) and equivalently $(v,\chi)$ the solution of (\ref{eqavtproj}). Then, there exist $M_1>0,$ $N_1>0$ and 
some positive constant $C$ such that for all $M\geq M_1,$ for all $\displaystyle T<\inf\{ T_{max},T(\a,M),\frac{M}{c}-N_1\}$ we have the following estimates:
\begin{equation}
\lV R_3\rV_{C^{0,\d}([0,T]\times\Omega^\infty)} + \lV R_4\rV_{C^{1,\d}([0,T]\times\partial\Omega^\infty)} \leq C e^{2\a(cT-M)}.
\end{equation}
\end{lem} 
We need two steps to prove the above lemma. In the next paragraph we use lemma \ref{lemtail} to control the solution $u$ for large $x.$ 
Then we use these estimates to globally control $R_{1,2,3}.$

\paragraph{estimate of the tail of $u$ thanks to Lemma \ref{lemtail}}
Let us recall that as 0 and 1 are respectively sub and supersolution for (\ref{cauchygros}), $0<u<1$ for all $t>0,(x,y)\in\Omega.$ From (\ref{estimeesbistable}) we have
for any $K_0>0$
\begin{align}
 u(t=0,x,y) = & \vp(x+M) \nonumber \\
  \leq & C_2 e^{\la\lp x+M-K_0\rp} e^{\la K_0}. \nonumber
\end{align}
Let us choose $K_0$ such that $\displaystyle C_2 e^{\la K_0}<\frac{\e_2}{2}$ and $M_1>K_0.$ Knowing that $0<\a_1<-\la,$ we can apply Lemma \ref{lemtail}.
It yields
\begin{equation*}
\begin{array}{c}
 0<u(t,x,y)<\lV \psi\rV_\infty e^{-\a_1\lp x-ct+M-K_0\rp}, \\ 
 \forall x>ct+K_0-M,\ \forall t\in[0,\min \{T_{max},T(\a,M),\frac{M}{c}-N_1\}]
\end{array}
\end{equation*}
where $\displaystyle N_1=\frac{1}{\a c}\log(\frac{2}{\e_2})+1.$ Indeed, this choice of $N_1$ ensures that $|v|<\frac{\e_2}{2},$ hence the inequality  $\displaystyle u(t,ct-M+K_0,y)<\e_2\leq \psi(y)$ 
is valid on the desired interval in time and the function $(t,x,y)\mapsto \psi(y)e^{-\a_1(x-ct+M-K_0)}$ is strictly above $u$ on the considered domain.
This provides a $L^\infty$ estimate for $u.$ With the regularity of $f,$ $u$ satisfies a parabolic equation with Neumann boundary condition and a second term satisfying the same $L^\infty$ estimate.
Hence, with classical parabolic estimates (see \cite{Lady}, theorem 10.1 p.204 for the Hölder regularity, then theorem 10.1 p.351 for the $C^{2,\d}$ regularity)
we have for some constant $C$ and some $\d\in(0,1)$
\begin{equation*}
 \begin{array}{c}
 \lV u\rV_{C^{2,\d}\lp [0,T]\times \Omega\cap\{x>X\}\rp} \leq C e^{-\a_1\lp X-cT+M-K_0\rp}, \\
 \forall T<\min \{T_{max},T(\a,M),\frac{M}{c}-N_1\},\forall X>cT+K_0-M+1.
 \end{array}
\end{equation*}
Now, $\vec{\Phi}$ is a $C^3-$diffeomorphism, hence a diffeomorphism with regularity $C^{2,\d}$ for all $\d\in(0,1).$ So the above estimate works on $\ut.$ In the moving framework,
it yields for some constant $C$
\begin{equation}\label{queueutilde}
 \begin{array}{c}
 \lV \ut\rV_{C^{2,\d}\lp [0,T]\times \Omega^\infty\cap\{\xi>\tilde{X}\}\rp} \leq C e^{-\a_1\lp \tilde{X}+M-K_0\rp}, \\
 \forall T<\min \{T_{max},T(\a,M),\frac{M}{c}-N_1\},\ \forall \tilde{X}>K_0-M+1.
 \end{array}
\end{equation}

\paragraph{Global control of $R_3, R_4$}

In this paragraph we conclude the proof of lemma \ref{controlrestes}. We give details only for $R_3,$ the other being similar.
We deal with $R_1,$ defined by (\ref{reste1}) by combining 
the estimate (\ref{queueutilde}) with the assumption (\ref{dominationPhi}). The quadratic term $R_3-R_1$ defined in (\ref{Reste3}) 
is controlled by the definition of $T(\a,M)$ in (\ref{TaMgros}).

From our hypotheses (\ref{hypdomain1}-\ref{hypdomain2}) on $\Omega$ the function $\Phi$ and all its derivatives up to the third order 
are uniformly bounded in $\Omega.$
Similarly, as $0<u<1$ for all $t>0,$ $(x,y)\in\Omega,$ parabolic estimates provide a uniform bound for $\ut_\xi$ and $\nabla_z \ut.$
Let us recall that, with (\ref{Phitranslate}), the derivatives of $\Phi$ can be considered as functions of the $(\xi,z)-$variables.
Hypothesis (\ref{dominationPhi})
gives, translated in the moving framework coordinates
\begin{equation}\label{controlPhi}
 \lV D_x\Phi(.+ct,.)\rV_{C^2(\Omega\cap\{\xi<\tilde{X}\})}+ \lV D_y\Phi(.+ct,.)-Id\rV_{C^2(\Omega\cap\{\xi<\tilde{X}\})} \leq C e^{\kappa (\tilde{X}+ct)}.
\end{equation}
We update the value of $N_1$ and $M_1$ by
$$
N_1\longleftarrow\max\lp N_1,2\frac{K_0+1}{c}\rp,\ M_1\longleftarrow \max\lp M_1,cN_1\rp.
$$
Thus, for all $T<\min \{T_{max},T(\a,M),\frac{M}{c}-N_1\},$
we set in (\ref{queueutilde}) and (\ref{controlPhi})
$$\tilde{X}=-M+\frac{1}{2}(M-cT).$$	 It yields for some positive constant $C$
\begin{equation*}
 \lV R_1\rV_{C^{1,\d}\lp[0,T]\times\Omega^\infty\rp}\leq C\lp e^{-\frac{\a_1}{2}(M-cT)}+e^{-\frac{\kappa}{2}(M-cT)}\rp.
\end{equation*}
We recall that $\a$ is given by (\ref{defalpha}). Using the definition of $T(\a,M)$ given by (\ref{TaMgros}) and the regularity of $\vp$ and $f,$ up to a smaller $\d,$
the proof of Lemma \ref{controlrestes} is completed for $R_3.$ The other are similar.
\qed

\subsection{Equation on $\MR(\ML)$}
We project the system (\ref{eqavtproj}) on $\MR(\ML).$ Using Remark \ref{dzvonR} it provides the following system for $v$
\begin{equation}\label{eqsurv}
 \begin{cases}
 \partial_t v-\ML v-\Delta_z v = Q\lc R_3\rc & t>0,(\xi,z)\in\Omega^\infty \\
 \partial_{\nu^\infty} v = Q\lc R_4\rc & t>0, (\xi,z)\in\partial\Omega^\infty \\
 v(0,\xi,z) = 0.
 \end{cases}
\end{equation}
We state a lemma similar to Lemma \ref{lemmeimage} given in the previous section.

\begin{lem}\label{Linftyimage}
Let $v$ be the solution of (\ref{eqsurv}). There exists a positive constant $C$ such that for all $M>M_1,$ for all $T<\min \{T_{max},T(\a,M),\frac{M}{c}-N_1\},$
$$
\lV v\rV_{C^{2,\d}\lp[0,T]\times\Omega^\infty\rp} \leq C e^{2\a(cT-M)}
$$
where $\a,M_1,N_1$ are given by Lemma \ref{controlrestes}.
\end{lem}

\begin{proof}
Once again, the proof lies on parabolic estimates. The main point is to get the estimate in the $L^\infty-$norm. Then, using the 
continuity of $Q$ with respect to the Hölder norms, the result follows from Theorem 10.1 p.351 in \cite{Lady}.

For all $t>0$ let $G(t,\xi,z)$ be an extension of $Q[R_4]$ in $\Omega^\infty,$ \textit{i.e.} (see Theorem 0.3.2 in \cite{Lunardi} for instance)
\begin{equation}\label{ExtensionOp}
 \begin{cases}
 \lV G(t,.)\rV_{C^{2,\d}(\Omega^\infty)} \leq C \lV Q[R_4](t,.)\rV_{C^{1,\d}(\partial\Omega^\infty)}, \qquad t\geq 0 \\
 \partial_{\nu^\infty}G = Q[R_4],\qquad t\geq 0, \ (\xi,z)\in\partial \Omega^\infty.
 \end{cases}
\end{equation}

Up to consider $Q[G]$ instead of $G$ we can assume that $G(t)\in\MR(\ML),$ and the whole equation (\ref{eqsurv}) stays in the space $\MR(\ML).$
Let $e^{t\Delta_z}$ be the analytic semi-group generated by the Laplace operator in the $z-$direction with Neumann boundary conditions. 
It is well-known (see \cite{Lady}) that it is a bounded operator in $L^\infty$ as well as in $C^{2,\d}.$ Moreover, from 
remark \ref{dzvonR}, we have that $e^{t(\Delta_z+\ML)} = e^{t\Delta_z}e^{t\ML} = e^{t\ML}e^{t\Delta_z}.$
Finally, spectral considerations give that the decay property (\ref{decML}) is valid for the $L^\infty-$norm as well as for 
the $C^2-$norm.

The Duhamel's formula for (\ref{eqsurv}) gives (see Theorem 5.1.17 in \cite{Lunardi})
\begin{align}
 v(t) = & e^{t(\Delta_z+\ML)}G(0) + \int_0^t e^{(t-s)(\Delta_z+\ML)}\lp Q[R_3](s)+\lp \Delta_z+\ML\rp G(s)\rp ds \nonumber \\
  & - \lp\Delta_z+\ML\rp \int_0^t e^{(t-s)(\Delta_z+\ML)}\lp G(s)-G(0)\rp ds+G(0). \label{Duhamelv}
\end{align}
Applying lemma \ref{controlrestes} and the boundedness of the operators $\ML$ and $\Delta_z$ from $C^2(\Omega^\infty)$ onto $BUC(\Omega^\infty)$ 
the equation (\ref{Duhamelv}) gives the following $L^\infty$ estimates: \\
for all $\displaystyle t < \min \{T_{max},T(\a,M),\frac{M}{c}-N_1\}$ we have
\begin{align}
\lV v(t)\rV_{L^\infty(\Omega^\infty)} & \leq Ce^{2\a(ct-M)}\lp 1+e^{-\rho t}\rp+C\int_0^t e^{\rho(s-t)}e^{2\a(cs-M)}ds \nonumber \\
  & \leq  Ce^{2\a(ct-M)}.\label{estimeev}
\end{align}
The above control (\ref{estimeev}) combined with parabolic estimates concludes the proof of Lemma \ref{Linftyimage}.

\end{proof}

\subsection{Equation on $\MN(\ML)$}
We project the system (\ref{eqavtproj}) on $\MN(\ML).$ It yields the following system for $\chi:$
\begin{equation}\label{eqsurchi}
 \begin{cases}
 \partial_t \chi -\Delta_z\chi = \langle e_*,R_3\rangle , & t>0, z\in\omega^\infty \\
\partial_{\nu^\infty}\chi = \langle e_*,R_4\rangle , & t>0, z\in\partial \omega^\infty \\
 \chi(0,z) = 0.
 \end{cases}
\end{equation}
We have the following lemma:
\begin{lem}\label{Linftynoyau}
 Let $\chi$ be the solution of (\ref{eqsurchi}). There exists a positive constant $C$ such that for all $M>M_1,$ for all 
 $T<\min\left\{ T_{max},T(\a,M),\frac{M}{c}-N_1\right\},$
 \begin{equation*}
 \lV\chi\rV_{C^{2,\d}\lp[0,T]\times\omega^\infty\rp} \leq Ce^{2\a(cT-M)} 
 \end{equation*}
 where $\a,M_1,N_1$ are given by Lemma \ref{controlrestes}.
\end{lem}

The proof is similar to the proof of Lemma \ref{Linftyimage} and we skip it.
The linear form $e_*$ is continuous with respect to the Hölder norms. Hence, form lemma \ref{controlrestes} there 
exists $g,G$ such that for all 
$t<\min\left\{ T_{max},T(\a,M),\frac{M}{c}-N_1\right\},$
\begin{gather*}
 \langle e_*,R_3(t,.,z)\rangle = g(t,z),\qquad  \langle e_*,R_4(t,.,z)\rangle = \partial_{\nu^\infty}G(t,z), \\
\lV g(t)\rV_{C^{2,\d}(\omega^\infty)} + \lV G(t)\rV_{C^{2,\d}(\omega^\infty)}\leq Ce^{2\a(ct-M)}.
\end{gather*}
A Duhamel's formula gives similar $L^\infty$ estimates and parabolic estimates conclude the proof.
\qed

%

\subsection{Conclusion of the proof of Proposition \ref{grosprop}}
We conclude in the same way as for the proof of Proposition \ref{toyprop}. Thanks to the two previous Lemmas, it is easy to construct $M_0\geq M_1,$ $N_0\geq N_1$ such that
for all $M\geq M_0,$ $T_{max} = T(\a,M) = \frac{M}{c}-N_0.$ Stating $\gamma=\a,$ the proof is finished.
\qed	

\section*{Remarks and questions}
\begin{itemize}
 \item In theorem \ref{solentieretoy}, we do not prove the uniqueness of such an entire solution.
 However, if one is able to prove the monotonicity in time of the solution, or that these solutions are transition fronts results from \cite{BHM} or
 \cite{BHgeneralisedTW} may provide the uniqueness.
 
 \item In both theorems we ask for an exponential decay of the perturbation, but it is arbitrary. Hence a natural guess is that it is not optimal 
 and some weaker convergence rate may be sufficient. But the proof also suggests that a too slow convergence may prevent the existence of such entire solutions.
%
 
 \item Most of the properties proved in \cite{BBC} are still valid for these domains. The invasion properties, theorems 1.5, 1.7, 1.8, 1.9 apply in our context.
 The blocking property, theorem 1.6, is more intricate but can easily be adapted.
\end{itemize}

\newpage 
\bibliographystyle{plain}
\footnotesize
\bibliography{biblio_generale}
\end{document}